\documentclass[graybox]{svmult}
\usepackage{bm}
\usepackage{mathptmx, amsfonts}

\usepackage{graphicx}
\usepackage[bottom]{footmisc}

\usepackage[backrefs, alphabetic]{amsrefs}

\newcommand{\GL}{{\rm GL}} \newcommand{\Tr}{{\rm Tr}}
\newcommand{\G}{{\rm G}} \newcommand{\T}{{\rm T}}
\newcommand{\K}{{\rm K}}

\newcommand{\la}{\lambda} \newcommand{\al}{\alpha} \newcommand{\be}{\beta} \newcommand{\ga}{\gamma} \newcommand{\de}{\delta}

\newcommand{\bbC}{\mathbb{C}} \newcommand{\bbN}{\mathbb{N}}
\newcommand{\bbZ}{\mathbb{Z}}
\newcommand{\bbD}{\mathbb{D}} 

\newcommand{\mcH}{\mathcal{H}}
\newcommand{\mcA}{\mathcal{A}}
\newcommand{\mcB}{\mathcal{B}}
\newcommand{\mcO}{\mathcal{O}}

\newcommand{\bft}{{\bm{t}}} \newcommand{\bfd}{{\bm{d}}}

\newcommand{\bfmu}{{\bm{\mu}}} \newcommand{\bfn}{{\bm{n}}}

\newcommand{\bfX}{{\bm{X}}}
\newcommand{\bfx}{{\bm{x}}} \newcommand{\bfy}{{\bm{y}}}

\newcommand{\fg}{\mathfrak{g}}
\newcommand{\fgl}{\mathfrak{gl}}
\newcommand{\fh}{\mathfrak{h}}
\newcommand{\zg}{\mathfrak{z(g)}}
\newcommand{\ot}{\otimes}

\renewcommand{\c}[2]{c^{#1}_{#2}}
\newcommand{\F}[2]{{\rm F}^{#1}_{#2}}
\newcommand{\form}[1]{\left \langle #1 \right \rangle}
\newcommand{\dual}[1]{{\left( #1 \right)}^*}

\begin{document}

\titlerunning{Sums of squares of Littlewood-Richardson coefficients}

\title*{Sums of squares of Littlewood-Richardson coefficients and
$\GL_n$-harmonic polynomials}

\author{Pamela E. Harris \and Jeb F. Willenbring}
\institute{Pamela E. Harris \at Marquette University, MSCS Department \\
P.O. Box 1881, Milwaukee WI 53201 \\
\email{pamela.harris@marquette.edu}
\and Jeb F. Willenbring \at
University of Wisconsin-Milwaukee, Dept. of Math. Sciences \\
P.O. Box 0413, Milwaukee, WI 53211 \\ \email{jw@uwm.edu} }

\maketitle

\abstract{We consider the example from invariant theory concerning the conjugation action of the general linear group on several copies of the $n \times n$ matrices, and examine a symmetric function which stably describes the Hilbert series for the invariant ring with respect to the  multigradation by degree.  The terms of this Hilbert series may be described as a sum of squares of Littlewood-Richardson coefficients.  A ``principal specialization'' of the gradation is then related to the Hilbert series of the $\K$-invariant subspace in the $\GL_n$-harmonic polynomials (in the sense of Kostant), where $\K$ denotes a block diagonal embedding of a product of general linear groups.  We also consider other specializations of this Hilbert series.
\footnote{This research was supported by the National Security Agency grant \# H98230-09-0054.}
}

\medskip
\centerline\obeylines{{\noindent \emph{To Nolan Wallach, who has influenced generations of mathematicians.}}}
\medskip

\section{Introduction}
\label{sec:Introduction}

We let $\F{\la}{n}$ denote the finite dimensional irreducible representation of the general linear group, $\GL_n$, (over $\bbC$) whose highest weight is indexed by the integer partition $\la = (\la_1 \geq \la_2 \geq \cdots \geq \la_n )$ (in standard coordinates).  Given a finite sequence of integer partitions ${\bfmu} = (\mu^{(1)}, \cdots, \mu^{(m)})$, we will let $\c{\la}{\bfmu}$ be the multiplicity of $\F{\la}{n}$ in the $m$-fold tensor product
$\F{\mu^{(1)}}{n} \ot \cdots \ot \F{\mu^{(m)}}{n}$ under the diagonal action of $\GL_n$.  That is, $\c{\la}{\bfmu}$ denotes a (generalized) Littlewood-Richardson coefficient.

We remark that the usual exposition of Littlewood-Richardson coefficients (see \cite{Fulton, GW, Howe-Lee-1, Howe-Lee-2, HTW, Macdonald}) concerns the case where $m=2$.  However, by iterating the Littlewood-Richardson rule (or its equivalents) one obtains several  effective combinatorial interpretations of our $\c{\la}{\bfmu}$.

The subject of this exposition concerns some interpretations of the positive integer
$\sum {\left( \c{\la}{\mu} \right)}^2$ where the sum is over certain finite subsets of non-negative integer partitions.  We believe that such sums have under-appreciated combinatorial significance.  For example, one immediately observes the very simple specialization to the case where $\mu^{(j)} = (1)$ for all $j=1,\cdots,m$, in which case the sum of squares is $m!$, which may be viewed as a consequence of Schur-Weyl duality.  More generally, if $\nu = (\nu_1 \geq \cdots \geq \nu_m \geq 0)$ is a partition and $\mu^{(j)} = (\nu_j)$, then $\c{\la}{\bfmu}$ is equal to the Kostka number $K_{\la \nu}$ (i.e. the multiplicity of the weight indexed by $\nu$ in $\F{\la}{n}$).

Our motivation for considering these numbers comes from invariant theory.  On one hand, we consider the conjugation action of the general linear group on several copies of the $n \times n$ matrices.  On the other hand, we consider the $\K$-conjugation action on one copy of the $n\times n$ matrices, where $\K$ denotes a block diagonal embedding of a product of general linear groups.  These problems are related, and have been studied extensively.  We make no attempt to survey the literature, but recommend \cite{Drensky}.

Central to this work is the notion of a \emph{Hilbert series}.  Let $V$ be a graded vector space.  That is, $V = \bigoplus_{d=0}^\infty V_d$ where $V_d$ is a finite dimensional subspace.  The Hilbert series, $\sum_{d=0}^\infty (\dim V_d) q^d$, formally records the dimensions of the graded components.  Here $q$ is an indeterminate.  We also consider multivariate generalizations corresponding to situations where $V$ is graded by a cone in a lattice.

In our setting, $V$ will be a space of invariant polynomial functions on $m$ copies of the $n \times n$ matrices.  For fixed values of the parameters, the Hilbert series is the Taylor expansion of a rational function around zero.  When these parameters are small, one can expect to write down the numerator and denominator explicitly.  These polynomials encode structural information about the invariants.  However, as the size of the matrix becomes large, these rational functions are difficult to compute explicitly.  This motivates re-organizing the data by studying the coefficients of the Hilbert series of a fixed degree as the size of the matrix goes to infinity.  The limit exists.   The formal series recording this information will be referred to as the \emph{stable Hilbert series}.

Certain sums of squares of Littlewood-Richardson coefficients describe the coefficients of the (stable) Hilbert series for the invariant algebra in each case.  These Hilbert series, stably, may be expressed as a product.  Furthermore, a ``principal specialization'' of this product is then related to the Hilbert series of the $\K$-invariant subspace in the $\GL_n$-harmonic\footnote{Harmonic in the sense of Kostant (see \cite{Kostant}), which generalizes the usual notion of a harmonic polynomial.} polynomials.

Unless otherwise stated, we will only need notation for representations with polynomial matrix coefficients, which are indexed by partitions with \emph{non-negative} integer components.  The sum of the parts of a partition $\nu$ will be called the \emph{size} (denoted $|\nu|$), while the number of parts will be called the \emph{length} denoted $\ell(\nu)$.  As usual, we will also write $\la \vdash d$ to mean $|\la| = d$.  Furthermore, we also adapt the (non-standard) notation that if $\bfmu = (\mu^{(1)}, \cdots, \mu^{(m)})$ is a finite sequence of partitions, we set $|\bfmu| = \sum_{j=1}^m |\mu^{(j)}|$, and write $\bfmu \vdash d$ to mean $|\bfmu|=d$.  From a combinatorial point of view, the results involve a specialization of the following

\begin{theorem}[Main Formula]  Let $t_1, t_2, t_3, \cdots$ denote a countably infinite set of indeterminates.  We have
\[
    \prod_{k=1}^\infty \frac{1}{1- \left(t_1^k + t_2^k + t_3^k + \cdots \right)}  = \sum_\la \sum_\bfmu {\left( \c{\la}{\bfmu} \right)}^2 \bft^\bfmu
\]
where the outer sum is over all partitions $\la$ and the inner sum is over all finite sequences of partitions $\bfmu = (\mu^{(1)}, \mu^{(2)}, \mu^{(3)}, \cdots)$ with $\bft^\bfmu = t_1^{|\mu^{(1)}|} t_2^{|\mu^{(2)}|} t_3^{|\mu^{(3)}|} \cdots$.
\end{theorem}
\begin{proof} See Section \ref{sec:main_formula}. \qed
\end{proof}

As an application of the main formula, we turn to the space, $\mcH(\fgl_n)$, of $\GL_n$-harmonic polynomial functions on the adjoint representation (with its natural gradation) by polynomial degree, $\mcH(\fgl_n) = \bigoplus_{d = 0}^\infty \mcH^d(\fgl_n)$.
The group $\GL_n$ acts on $\mcH(\fgl_n)$.  Note that the constant functions are the only $GL_n$-invariant harmonic functions.  However, if $\K$ is a reductive subgroup of $\GL_n$, the space of $\K$-invariant functions is much larger.  Consider the example when the group $\K$ is the block diagonally embedded copy of $\GL_{n_1} \times \cdots \times \GL_{n_m}$ in $\GL_n$, with $n_1+\cdots+n_m=n$. We will denote this group by $\K(\bfn)$ where $\bfn = (n_1, \cdots, n_m)$.  The purpose of this paper is to relate the dimension of the $\K(\bfn)$-invariant polynomials in $\mcH^d(\fgl_n)$ to a sum of squares of Littlewood-Richardson coefficients.  See Theorem \ref{thm:main} for the precise statement.

We consider this question because the related algebraic combinatorics are particularly elegant, and hence have expository value in connecting harmonic analysis with algebraic combinatorics.  However, this example is the tip of an iceberg.  Indeed, one can replace $\GL_n$ with any algebraic group G (with $\fg = Lie(G)$) and $\K(\bfn)$ with any subgroup of $G$.  This area of investigation is wide open and well motivated as an examination of the special symmetries of harmonic polynomials.

In Section \ref{sec:Hesselink}, we describe a general ``answer'' to this question when K is a symmetric subgroup of a reductive group G.  This answer is not as explicit as would be desired, but applies to any symmetric pair (G,K).  The remainder of the paper is related to the $\G = \GL_n$ example with $\K = \K(\bfn)$.  Note that when $m=2$ (i.e. $n = (n_1, n_2)$) the pair (G,K) is symmetric, but for $m > 2$ is not.  We remark that in the $m=2$ case, the results presented here were first described in \cite{Willenbring}.  Our present discussion amounts to a generalization to $m>2$.

After setting up appropriate notation in Section \ref{sec:Preliminaries} we provide an interpretation for a description of the Hilbert series of the $\K(\bfn)$-invariants in the $\GL_n$-harmonic polynomials on $\fgl_n$ in Sections \ref{sec:GLn-invariants} and \ref{sec:GLn-harmonics}.  Chief among these involves sums of squares of Littlewood-Richardson coefficients.  We recall other combinatorial interpretations in Section \ref{sec:combinatorics}.  These interpretations involve counting the conjugacy classes in the general linear group over a finite field.

\bigskip
\noindent {\bf Acknowledgements}  The first author wishes to thank Marquette University for support during the preparation of this article.  The second author wishes to thank the National Security Agency for support.  We also would like to thank both Lindsey Mathewson and the referee for pointing out several references that improved the exposition.

\section{The case of a symmetric pair}\label{sec:Hesselink}

Let G denote a connected reductive linear algebraic group over the complex numbers and let $\fg$ be its complex Lie algebra.  We have $\fg = \zg \oplus \fg_{ss}$, where $\zg$ denotes the center of $\fg$, while $\fg_{ss} = [\fg, \fg]$ denotes the semisimple part of $\fg$.  A celebrated result of Kostant (see \cite{Kostant}) is that the polynomial functions on $\fg$, denoted $\bbC[\fg]$, are a free module over the invariant subalgebra, $\bbC[\fg]^G$, under the adjoint action.  Choose a Cartan subalgebra $\fh$ of $\fg$, and let $\Phi$ and $W$ denote the corresponding set of roots and Weyl group, respectively.  Choose a set of positive roots $\Phi^+$, and let $\Phi^- = - \Phi^+$ denote the negative roots.  Set $\rho = \frac{1}{2} \sum_{\al \in \Phi^+} \al$.  For $w \in W$, let $l(w)$ denote the number of positive roots sent to negative roots by $w$.  Fix an indeterminate $t$.  There exist positive integers $e_1 \leq e_2 \leq \cdots \leq e_r$ such that $\sum_{w \in W} t^{l(w)} = \prod_{j=1}^r \frac{1-t^{e_j}}{1-t}$ where $r$ is the rank of $\fg_{ss}$.  A consequence of the Chevalley restriction theorem (\cite{Chevalley}) is that $\bbC[\fg]^G$ is freely generated, as a commutative ring, by  $\dim \zg$ polynomials of degree 1, and $r$ polynomials of degree $e_1, \cdots, e_r$. These polynomials are the \emph{basic invariants}, while $e_1, \cdots, e_r$ are called the \emph{exponents} of $G$.

\subsection{Harmonic polynomials}\label{subsec:Harmonics}

We define the \emph{harmonic polynomials} on $\fg$ by
\[
    \mcH_\fg = \left\{ f \in \bbC[\fg] |  \Delta(f) = 0 \mbox{ for all } \Delta \in \bbD[\fg]^G \right\}
\] where $\bbD[\fg]^G$ is the space of constant coefficient G-invariant differential operators on $\fg$.  In \cite{Kostant}, it is shown that as a G-representation $\mcH_\fg$ is equivalent to the $G$-representation algebraically induced from the trivial representation of a maximal algebraic torus, $\T$, in $\G$.  Thus, by Frobenius reciprocity the irreducible rational representations of $\G$ occur with multiplicity equal to the dimension of their zero weight space.  Moreover, as a representation of G, the harmonic polynomials are equivalent to the regular functions on the nilpotent cone in $\fg$.

The harmonic polynomials inherit a gradation by degree from $\bbC[\fg]$.  Set $\mcH_\fg^d = \mcH_\fg \cap \bbC[\fg]_d$.  Thus, $\mcH_\fg = \bigoplus_{d=0}^\infty \mcH_\fg^d$.  We next consider the distribution of the multiplicity of an irreducible G-representation among the graded components of $\mcH_\fg$.  A solution to this problem was originally due to Hesselink \cite{Hesselink}, which we recall next.

Let $\wp_t:\fh^* \rightarrow \bbN$ denote\footnote{As always, $\bbN = \{0,1,2,3,\cdots$\}, the non-negative integers.} Lusztig's q-analog of Kostant's partition function.  That is $\wp_t$ is defined by the equation,
\[
    \prod_{\alpha \in \Phi^+} \frac{1}{1 - t e^\alpha} = \sum_{\xi \in Q(\fg, \fh)} \wp_t(\xi) e^\xi
\] where $Q(\fg,\fh) \subseteq \fh^*$ denotes the lattice defined by the integer span of the roots.  As usual, $e^\xi$ denotes the corresponding character of $\T$, with $Lie(\T) =\fh$.  As usual, we set $\wp(\xi) = 0$ for $\xi \notin Q(\fg, \fh)$.

Let $P(\fg)$ denote the integral weights corresponding to the pair ($\fg, \fh$).  The dominant integral weights corresponding to $\Phi^+$ will be denoted $P_+(\fg)$.  Let $L(\la)$ denote the (finite dimensional) irreducible highest weight representation with highest weight $\la \in P_+(\fg)$.  The multiplicity of $L(\la)$ in the degree $d$ harmonic polynomials, $\mcH^d(\fg)$, will be denoted by $m_d(\la)$.  Set $m_\la(t) = \sum_{d=0}^\infty m_d(\la) t^d$.  Hesselink's theorem asserts that
\[
    m_\la(t) = \sum_{w \in W} (-1)^{l(w)} \wp_t( w(\la + \rho) - \rho).
\]  See \cite{Wallach-Willenbring} for a generalization of this result.

We remark that the above formula is very difficult to implement in practice.  This is in part due to the fact that the order of $W$ grows exponentially with the Lie algebra rank.  Thankfully, only a small number of terms actually contribute to the overall multiplicity.  See \cite{Harris} for a very interesting special case where the number of contributed terms is shown to be a Fibonacci number.

\subsection{The K-spherical representations of G}

Let (G,K) be a symmetric pair.  That is, G is a connected reductive linear algebraic group over $\bbC$ and $\K = \{ g \in \G| \theta(g) = g \}$, where $\theta$ is a regular automorphism of G of order two.  Since $\K$ will necessarily be reductive the quotient, $\G/\K$ is an affine variety and the $\bbC$-algebra of regular function $\bbC[\G/\K]$ is multiplicity free as a representation of $G$.  This fact follows from the (complexified) Iwasawa decomposition of $G$.  Put another way, there exists $S \subseteq P_+(\fg)$ such that for all $\la \in P_+(\fg)$ we have
\[ \dim L(\la)^\K = \left\{
     \begin{array}{ll}
       1, & \hbox{$\la \in S$;} \\
       0, & \hbox{$\la \not\in S$.}
     \end{array}
   \right.
\]  Note that the subset $S$ may be read off of the data encoded in the Satake diagram associated to the pair (G,K).   The above fact implies that the Hilbert series $H_t(G,K) = \sum_{d=0}^\infty h_d t^d$ with $h_d = \dim \mcH^d(\fg)^\K$ has the following formal expression
\[
    H_t(G,K) = \sum_{w \in W} (-1)^{l(w)}
    \left( \sum_{\la \in S} \wp_t(w(\la+\rho)-\rho) \right).
\]  This formula seems rather encouraging.  Unfortunately, the inner sum is very difficult to put into a closed form for general $w \in W$.  This is, in part, a reflection of the fact that the values of $\wp_t$ cannot be determined from a ``closed form'' expression.  However, note that $w(\la + \rho)-\rho$ often falls outside of the support of $\wp_t$, and therefore it may be possible to obtain explicit results along these lines.  Moreover, the point of this exposition is to advertise that combinatorially elegant expressions may exist.  At least this is the case for the pair ($\GL_{n_1+n_2}, \GL_{n_1} \times \GL_{n_2}$), as we shall see.

\section{Preliminaries}\label{sec:Preliminaries}

We let $\fgl_n$ denote the complex Lie algebra of $n \times n$ matrices with the usual bracket, $[X, Y] = XY - YX$, for $X, Y \in \fgl_n$.  Let $E_{ij}$ denote the $n \times n$ matrix with a 1 in row $i$ and column $j$, and a 0 everywhere else.  The Cartan subalgebra will be chosen to be the span of $\{E_{ii} | 1 \leq i \leq n \}$.  The dual basis in $\fh^*$ to $(E_{11}, E_{22}, \cdots, E_{nn})$ will be denoted $(\epsilon_1, \cdots, \epsilon_n)$.  Choose the simple roots as usual, $\Pi = \{\epsilon_i - \epsilon_{i+1} | 1 \leq i < n \}$.  Let $\Phi$ (resp. $\Phi^+$) denote the roots (resp. positive roots).  We will identify $\fh$ with $\fh^*$ using the trace form, $(H_1, H_2) = \Tr(H_1 H_2)$ (for $H_1, H_2 \in \fh)$.
The fundamental weights are $\omega_i = \sum_{j=1}^i \epsilon_j \in \fh^*$ for $1 \leq i \leq n-1$.  We also set $\omega_n = \sum_{j=1}^n \epsilon_j \in \fh^*$.  Let $P(\GL_n) = \sum_{j=1}^n \bbZ \omega_j$, and $P_+(\GL_n) = \bbZ \omega_n + \sum_{j=1}^{n-1} \bbN \omega_j$.

From this point on, we will write $(a_1, \cdots, a_n)$ for $\sum a_i \epsilon_i$.  Thus, we have $\la = (\la_1,\cdots,\la_n) \in P_+(\GL_n)$ iff each $\la_i$ in an integer and $\la_1 \geq \cdots \geq \la_n$.  The finite dimensional, irreducible representation of $\GL_n$ with highest weight $\la$ will be denoted $(\pi_\la, \F{\la}{n})$ where
\[ \pi_\la : \GL_n \rightarrow \GL(\F{\la}{n}).\]
To simply notation will write $\F{\la}{n}$ for $(\pi_\la, \F{\la}{n})$.

Throughout this article, the representations of $\GL_n$ which we will consider have polynomial matrix coefficients.  Thus the components of the highest weight $\la$ will be non-negative integers.  Therefore, if $\la$ is a (non-negative integer) partition with at most $\ell$ parts ($\ell \leq n$), then the $n$-tuple, $(\la_1, \cdots, \la_\ell, 0, \cdots, 0)$, corresponds to the highest weight of a finite dimensional irreducible representation of $\GL_n$ (with polynomial matrix coefficients).

Under the diagonal action, $\GL_n$ acts on the $d$-fold tensor product $\otimes^d \bbC^n$.  Schur-Weyl duality (see \cite{GW} Chapter 9) asserts that the full commutant to the $\GL_n$-action is generated by the symmetric group action defined by permutation of factors.  Consequently, one has a multiplicity free decomposition with respect to the joint action of $\GL_n \times S_d$.  Moreover, if $n \geq d$ then every irreducible representation of $S_d$ occurs.  The irreducible representation of $S_d$ paired with $\F{\la}{n}$ will be denoted by $V^{\la}_{d}$. The full decomposition into irreducible $GL_n \times S_d$-representation is
\[
    {\bigotimes}^d \bbC^n \cong \bigoplus_\la \; \F{\la}{n} \otimes V^{\la}_{m}
\] where the sum is over all non-negative integer partitions, $\la$, of size $d$ and length at most $n$.  Note that when $n = d$, then the condition on $\ell(\la)$ is automatic.  Thus, all irreducible representations of $V^{\la}_d$ occur.  In this manner, the highest weights of $\GL_n$-representations provide a parametrization of the $S_d$-representations.

\subsection{Littlewood-Richardson coefficients}

Let $\bfd = (d_1, \cdots, d_m)$ denote a tuple of positive integers with $d=d_1+\cdots+d_m$.  Let $S_{\bfd}$ denote the subgroup of $S_d$ consisting of permutations that stabilize the sets permuting the first $m_1$ indices, then the second $m_2$ indices etc.  Clearly, we have
\[
    S_\bfd \cong S_{d_1} \times \cdots \times S_{d_m}.
\] The irreducible representations of $S_\bfd$ are of the form
\[
    V(\bfmu) = V^{\mu^{(1)}}_{d_1} \otimes \cdots \otimes V^{\mu^{(m)}}_{d_m}
\] where $\mu^{(j)}$ is a partition of size $d_j$.  It is well known that if an irreducible representation, $V^{\la}_d$, of $S_d$ is restricted to $S_\bfd$ then the multiplicity of $V(\bfmu)$ in $V^{\la}_d$ is given by the Littlewood-Richardson coefficient $\c{\la}{\bfmu}$. This fact is a consequence of Schur-Weyl duality.

\section{Invariant polynomials on matrices}\label{sec:GLn-invariants}
A permutation of $\{1,2, \cdots, m \}$ may be written as a product of disjoint cycles.  This result is fundamental to combinatorial properties of the symmetric group, $S_m$.  Keeping this elementary fact in mind, let $\bfX = (X_1, X_2, \cdots, X_m)$ be a list of complex $n \times n$ matrices.  Let $\Tr$ denote the trace of a matrix and define
\begin{equation}
    \Tr_\sigma(\bfX) =   \Tr(X_{\sigma^{(1)}_1} X_{\sigma^{(1)}_2} \cdots X_{\sigma^{(1)}_{m_1}})
        \; \Tr(X_{\sigma^{(2)}_1} X_{\sigma^{(2)}_2} \cdots X_{\sigma^{(2)}_{m_2}})   \cdots
           \Tr(X_{\sigma^{(k)}_1} X_{\sigma^{(k)}_2} \cdots X_{\sigma^{(k)}_{m_k}})
\end{equation} where
$\sigma =
(\sigma^{(1)}_1 \sigma^{(1)}_2 \cdots \sigma^{(1)}_{m_1})
(\sigma^{(2)}_1 \sigma^{(2)}_2 \cdots \sigma^{(2)}_{m_2})
\cdots
(\sigma^{(k)}_1 \sigma^{(k)}_2 \cdots \sigma^{(k)}_{m_k})$ is a permutation of $m = \sum m_i$ written as a product of $k$ disjoint cycles.  Observe that the cycles of $\sigma$ may be permuted and rotated without changing the permutation.  In turn, the function $\Tr_\sigma$ displays these same symmetries.

Chief among the significance of $\Tr_\sigma$ is the fact that they are invariant under the conjugation action
$
    g \cdot (X_1, \cdots, X_m) = (g X_1 g^{-1}, \cdots, g X_m g^{-1}),
$   where $g \in \GL_n$.
By setting some of the components of $(X_1, \cdots, X_m)$ equal one defines a polynomial of equal degree but on fewer than $m$ copies of $M_n$.  Intuitive, this fact may be described by allowing equalities in the components of the cycles of $\sigma$.  That is to say (formally) we consider $\sigma$ up to conjugation by Levi subgroup of $S_m$.

In \cite{Procesi-AMS, Procesi-AIM}, C. Procesi described these generators for the algebra of $\GL_n$-invariant polynomials, denoted\footnote{Here we let $M_n$ denote the complex vector spaces of $n \times n$ matrices (with entries from $\bbC$).} $\bbC[V]^{\GL_n}$, on $V = M_n^m$, and provided a proof that these polynomials span the invariants.  Hilbert tells us that the ring of invariants must be finitely generated.  Thus, there must necessarily be algebraic relations among this (infinite) set of generators.  In light of Procesi's work, these generators and relations are understood.  However, recovering the Hilbert series from these data is not automatic.

In order to precisely quantify the failure of the $\Tr_\sigma$ being independent, we introduce the formal power series $A_n(\bft) = A_n(t_1, t_2, \cdots, t_m) = \sum a_n(\bfd) \bft^\bfd$ with
\footnote{We will use the notation $\bfd = (d_1, \cdots, d_m)$, $\bft^\bfd = t_1^{d_1} \cdots t_m^{d_m}$.} the coefficient defined as  $a_n(\bfd) = \dim \bbC[M_n^m]^{\GL_n}_\bfd$ where $\bbC[M_n^m]_\bfd = \bbC[M_n \oplus \cdots \oplus M_n]_{(d_1, \cdots, d_m)}$ denotes the homogeneous polynomials of degree $d_i$ on the $i$th copy of $M_n$.  The multivariate series, $A_n(\bft)$, is called the \emph{Hilbert Series} of the invariant ring.  Except in some simple cases, a closed form for $A_n(\bft)$ is not known.  Part of the present exposition is to point out the rather simple fact that $a_n(\bfd)$ may be expressed in terms of the squares of Littlewood-Richardson coefficients.  Furthermore, we prove

\begin{proposition}\label{prop_mat_diag}  For any natural numbers $m$ and $\bfd = (d_1, \cdots, d_m)$ \\ the limit
$\lim_{n \rightarrow \infty} a_n(\bfd)$ exists.  If we call the limiting value $a(\bfd)$ and set $\widetilde A(\bft) = \sum a(\bfd) \bft^\bfd$ then
\[
    \widetilde A(\bft) = \prod_{k=1}^\infty \frac{1}{1- (t_1^k + t_2^k + \cdots + t_m^k) }.
\]
\end{proposition}

\begin{proof}
The polynomial functions on $M_n$ are multiplicity free under the action of $\GL_n \times \GL_n$ given by $(g_1, g_2)f(X) = f(g_1^{-1} X g_2)$ for $g_1, g_2 \in \GL_n$, $X \in M_n$ and $f \in \bbC[M_n]$.  The decomposition is a ``Peter-Weyl'' type, $\bbC[M_n] \cong \bigoplus \dual{\F{\la}{n}} \ot \F{\la}{n}$, where the sum is over all non-negative integer partitions $\la$ with $\ell(\la) \leq n$.  We have $\bbC[\oplus_{i=1}^m M_n] \cong \otimes_{i=1}^m \bbC[M_n]$.  Thus
\begin{eqnarray*}
    \bbC[M_n^m] &\cong&
      \left(  \bigoplus_{\mu^{(1)}} \dual{\F{\mu^{(1)}}{n}} \ot \F{\mu^{(1)}}{n}  \right)
           \ot \cdots \ot
      \left(  \bigoplus_{\mu^{(m)}} \dual{\F{\mu^{(m)}}{n}} \ot \F{\mu^{(m)}}{n}  \right) \\
    & \cong & \bigoplus_{\al, \be} \left( \sum_{\bfmu = (\mu^{(1)},\cdots,\mu^{(m)})}
        \c{\al}{\bfmu} \c{\be}{\bfmu}\right) \; \dual{\F{\al}{n}} \ot \F{\be}{n},
\end{eqnarray*}
with respect to the $\GL_n \times \GL_n$ action on the diagonal of $M_n^m$.  Note that in multidegree $(d_1, \cdots, d_m)$ the sum is over all $\bfmu$ with $|\bfmu^{(j)}| = d_j$.
If we restrict to the subgroup $\{ (g,g) : g \in \GL_n\}$ of $\GL_n \times \GL_n$, we obtain an invariant exactly when $\al = \be$.  The dimension of the $\GL_n$-invariants in the degree $d$ homogeneous polynomials on $M_n^m$ is therefore $\sum \left( \c{\la}{\bfmu} \right)^2$ where the sum is over all $\bfmu = (\mu^{(1)}, \cdots, \mu^{(m)})$ and $\la \vdash m$ with length at most $n$.  The degree $d$ component decomposes into multi-degree components $(d_1, \cdots, d_m)$ with $d = \sum d_j$.

If $d \geq n$, then the condition that $\ell(\la) \leq n$ is automatic, and this fact implies that if $\c{\la}{\bfmu} \neq 0$, then $\ell(\mu^{(j)})\leq n$ for all $j$.  Thus, the dimension of the degree $d$ invariants in $\bbC[M_n^m]$ is $\sum \left( \c{\la}{\bfmu} \right)^2$ where the sum is over \emph{all} partitions of size $d$.  If we specialize the main formula so that $t_j = 0$ for $j > m$ then the sums of squares of Littlewood-Richardson coefficients agree with $\widetilde A(\bft)$. \qed
\end{proof}

For our present purposes, we will specialize the multigradation on the invariants in $\bbC[M_n^m]$ to one that is more coarse.  From this process, we can relate the stabilized Hilbert series of the invariants in $\mcH[M_n]$.  This specialization will be the subject of the next section.

\bigskip
We now turn to another problem that, we shall see, is surprisingly related.  Consider the $n \times n$ complex matrix
\[
    X = \left[
          \begin{array}{cccc}
            X(1,1) & X(1,2) & \cdots & X(1,m) \\
            X(2,1) & X(2,2) & \cdots & X(2,m) \\
            \vdots & \vdots & \ddots & \vdots \\
            X(m,1) & X(m,2) & \cdots & X(m,m) \\
          \end{array}
        \right]
\] where $X(i,j)$ is an $n_i \times n_j$ complex matrix with $n = \sum n_j$.  Define
\[
    \Tr^\sigma(X) = \prod_{j=1}^k \Tr \left(
        X( \sigma^{(j)}_1,     \sigma^{(j)}_2 )
        X( \sigma^{(j)}_2,     \sigma^{(j)}_3 )
        X( \sigma^{(j)}_3,     \sigma^{(j)}_4 ) \cdots
        X( \sigma^{(j)}_{m_j}, \sigma^{(j)}_1 ) \right).
\]  Let $\K(\bfn)$ denote the block diagonal subgroup of $\GL_n$ of the form
\[
\K(\bfn) = \left[\begin{array}{ccc}
\GL_{n_1} &        &            \\
          & \ddots &            \\
          &        & \GL_{n_m}
\end{array}\right].
\]  The group $\K(\bfn)$ acts on $M_n$ by restricting the adjoint action of $\GL_n$.  The $\K(\bfn)$-invariant subring of $\bbC[M_n]$ (denoted $\bbC[M_n]^{\K(\bfn)}$) is spanned by $\Tr^\sigma(X)$.  For small values of the parameter space, these expressions are far from linearly independent (as in the last example).  Formally, one cannot help but notice the symbolic map $\Tr^\sigma \mapsto \Tr_\sigma$.  We will try next to make a precise statement along these lines.

Define $\bbC[M_n]_d$ to be the homogeneous degree $d$ polynomials on $M_n$, and let $\bbC[M_n]^{\K(\bfn)}_d$ denote the $\K(\bfn)$-invariant subspace.
Set $a^{(\bfn)}(d) = \dim \bbC[M_n]^{\K(\bfn)}_d$, and $A^{(\bfn)}(t) = \sum_{d=0}^\infty a^{(\bfn)}(d) \; t^d$.  Analogous to Proposition \ref{prop_mat_diag}, we have

\begin{proposition}\label{prop_mat_block}  For any non-negative integer $d$, the limit
\[
    \lim_{n_1 \rightarrow \infty} \lim_{n_2 \rightarrow \infty}
    \cdots \lim_{n_m \rightarrow \infty} a^{(n_1, \cdots, n_m)}(d)
\] exists.  Denote the limiting value $a(d)$ and set $A(t) = \sum_{d=0}^\infty a(d) t^d$.  We have
\[
    A(t) = \prod_{k=1}^\infty \frac{1}{1- m \; t^k}
\]
\end{proposition}
\begin{proof}
We begin with the $\GL_n \times \GL_n$-decomposition
$\bbC[M_n] = \bigoplus \dual{\F{\la}{n}} \ot \F{\la}{n}$ with respect to
the action in the proof of Proposition \ref{prop_mat_diag}.  The irreducible $\GL_n$-representation $\F{\la}{n}$ is reducible upon restriction to $\K(\bfn)$.  The decomposition is given in terms of Littlewood-Richardson coefficients
\[
    \F{\la}{n} \cong \bigoplus_{\bfmu=(\mu^{(1)}, \cdots, \mu^{(m)})} \c{\la}{\bfmu} \;
        \F{\mu^{(1)}}{n} \ot \cdots \ot \F{\mu^{(m)}}{n}.
\]
Therefore, as a $\K(\bfn) \times \K(\bfn)$-representation we have
\[
    \bbC[M_n]_d = \sum_{\bfmu, {\bf \nu}}
        \left( \sum_\la \c{\la}{\bfmu} \c{\la}{{\bf \nu}} \right) \;
        \dual{  \ot_{j=1}^m \F{\mu^{(j)}}{n_j} } \ot
        \left(  \ot_{j=1}^m \F{\nu^{(j)}}{n_j} \right)
\] where the sum is over all $\la$ with $|\la| = d$, $\ell(\la) \leq n$ and
       $\ell(\mu^{(j)}), \ell(\nu^{(j)}) \leq n_j$.  If we restrict to the diagonally
embedded $\K(\bfn)$-subgroup, we obtain an invariant exactly when $\bfmu = {\bf \nu}$.  Thus, $a^{(\bfn)}(d) = \sum \left( \c{\la}{\bfmu} \right)^2$, with the appropriate restrictions on $\la$ and $\bfmu$.

If all $n_j \geq d$ then the condition on the lengths of partitions disappears, and we may sum over all $\la \vdash d$.  The result follows by specializing the main formula by setting $t_j = t$ for $1 \leq j \leq m$ and $t_j = 0$ for $j>m$.  \qed
\end{proof}

Although we will not need it for our present purposes, it is worth pointing out that the algebra $\bbC[M_n]$ has a natural $\bbN^m$ gradation defined by the action of the center of $\K(\bfn)$.  This multigradation refines the gradation by degree.  The limiting multigraded Hilbert series is the same as that of Proposition \ref{prop_mat_diag}.  Upon specialization to $t_1 = t_2 = \cdots = t_m = t$ we obtain the usual gradation by degree (in both situations).  The advantage of considering the more refined gradation is that one can consider the direct limit as $m$ and $n$ go to infinity.  This will be relevant in the next section.

\section{Harmonic polynomials on matrices}\label{sec:GLn-harmonics}

A specific goal of this article is to understand the dimension of the space of degree $d$ homogeneous $\K(\bfn)$-invariant harmonic polynomials on $M_n$.  With this fact in mind, we observe the following specialization of the product in the main formula.  Let $t_j = t^j$.  Then we obtain:
\[
    \prod_{k=1}^\infty \frac{1}{1 - (t^k + t^{2k} + t^{3k} + \cdots)} =
    \prod_{k=1}^\infty \frac{1}{1- \frac{t^k}{1-t^k}} =
    \prod_{k=1}^\infty \frac{1-t^k}{1-2 t^k}.
\]
For a sequence $\bfmu = (\mu^{(1)}, \mu^{(2)}, \cdots)$ set $gr(\bfmu) = \sum_{j=1}^\infty j |\mu^{(j)}|$.  The equation in the main formula becomes
\[
    \prod_{k=1}^\infty \frac{1-t^k}{1- 2 t^k} =
    \sum_\la \sum_\bfmu \left( \c{\la}{\bfmu} \right)^2 t^{gr(\bfmu)}.
\]  The notation $gr$ is used for the word \emph{grade}.  We explain this choice next.  Let $\bbC[M_n; d]$ denote the polynomials functions on the $n \times n$ complex matrices together with the gradation defined by $d$ times the usual degree.  That is, $\bbC[M_n; d_1]_{(d_2)}$ denotes the degree $d_2$ homogeneous polynomials, but regarded as the $d_1 d_2$ graded component in $\bbC[M_n; d_1]$.  We consider the $\bbN$-graded complex algebra $\mcA_n$ defined\footnote{One might consider this as a ``Bosonic Fock space.''  Precisely, we mean the infinite tensor product in which each element is a finite sum of tensors involving only finitely many of the factors.} as
\begin{eqnarray*}
    \mcA_n &=& \bbC[M_n; 1] \otimes \bbC[M_n; 2] \otimes \bbC[M_n; 3] \otimes \cdots \\
           &=& \sum_{\de=0}^\infty \mcA_n[\de]
\end{eqnarray*} where $\mcA_n[\de]$ is the graded $\de\in \bbN$ component (with the usual grade defined on a tensor product of algebras).  The group $\GL_n$ acts on each $\bbC[M_n; d]$ by the adjoint action, and respects the grade.  Under the diagonal action on the tensors, the $\GL_n$-invariants, $\mcA_n^{\GL_n}$, have $A_n(\bft)$ as the Hilbert series when $\bft$ is specialized to $(t_1, t_2, t_3, \cdots) = (t, t^2, t^3, \cdots)$.  This $\GL_n$-action respects the $\de$ component in the gradation.  That is to say that the Hilbert series is
\[
    A_n(t,t^2,t^3, \cdots) = \sum_{\de=0}^\infty
          \dim \left(\mcA_n[\de]\right)^{\GL_n} t^\de.
\]
For $d=0,\cdots,n$ the coefficient of $t^d$ in $A_n(t,t^2,\cdots)$ is the same as the coefficient of $t^d$ in $\widetilde A(t,t^2,\cdots)$.
Summarizing, we can say that stably, as $n \rightarrow \infty$, the Hilbert series of $\mcA_n^{\GL_n}$ is $\prod_{k=1}^\infty \frac{1-t^k}{1-2t^k}$.

\bigskip
We turn now to the $\GL_n$-harmonic polynomials on $M_n$ together with its usual gradation by degree.  Kostant's theorem \cite{Kostant} tells us that
\[ \bbC[M_n] \cong \bbC[M_n]^{\GL_n} \ot \mcH(M_n).\]
As before let $\K = \K(n_1, n_2)$ denote the copy of $\GL_{n_1} \times \GL_{n_2}$ (symmetrically) embedded in $\GL_{n_1+n_2}$.  Passing to the $\K(n_1,n_2)$-invariant subspaces we obtain:
\[
    \bbC[M_n]^\K = \bbC[M_n]^{\GL_n} \ot \mcH(M_n)^\K.
\]  The Hilbert series of $\bbC[M_n]^{\GL_n}$ is well known to be $\prod_{k=1}^n \frac{1}{1-t^k}$, while the Hilbert series $\bbC[M_n]^\K$ is $A_n(t,t)$.  These facts imply that the dimension of the degree $d$ homogenous $\K$-invariant harmonic polynomials is the coefficient of $t^d$ in $F_n(t) = A_n(t,t) \prod_{j=1}^n (1-t^j)$.  We have that the coefficient of $t^d$ in $F_n(t)$ for \\
$d = 0, \cdots, \min(n_1, n_2)$ agrees with the coefficient of $t^d$ in
\[ \widetilde A(t,t^2, t^3, \cdots) \prod_{k=1}^\infty (1-t^k) = \prod_{k=1}^\infty \frac{1-t^k}{1-2 t^k}.
\]
Again summarizing, we say that stably, as $n_1, n_2 \rightarrow \infty$,\\
the Hilbert series of $\mcH(M_{n_1 + n_2})^{\K(n_1, n_2)}$ is the same as $\mcA_n^{\GL_n}$ as $n, n_1, n_2 \rightarrow \infty$.   That is to say, for \emph{fixed} $\de$ we have
\[
    \lim_{n_1,n_2 \rightarrow \infty} \dim \mcH(M_{n_1+n_2})^{\K(n_1, n_2)}_\de =
    \lim_{n \rightarrow \infty} \dim \mcA_n[\de]^{\GL_n}.
\]

\bigskip

Observe that this procedure generalizes.  Let $m \geq 2$.  Analogous to before, let $\bbC[M_n^m;d]$ denote the polynomial function on $M_n^m = M_n \oplus \cdots \oplus M_n$ ($m$-copies) together with the $\bbN$-gradation defined such that $\bbC[M_n^m;d_1]_{d_2}$ consists of the degree $d_2$ homogeneous polynomials but regarded as being the $d_1 d_2$-th graded component.  Note that the $\de$-th component in the grade is (0) if $\de$ is not a multiple of $d_1$.

Let $\mcA^m_n$ be the $\bbN$-graded algebra defined as
\[
    \mcA^m_n = \bbC[M_n^{m-1}; 1] \otimes \bbC[M_n^{m-1}; 2] \otimes \bbC[M_n^{m-1}; 3] \otimes \cdots
\]  Since $\mcA^m_n$ is a tensor product of $\bbN$-graded algebras, it has the structure of an $\bbN$-graded algebra.  As before let $\mcA^m_n[\de]$ is the $\de$-th graded component.  The group $\GL_n$ acts on each $\bbC[M_n^m; d]$ by the adjoint action, and respects the grade.

Next, set $t_1 = t_2 = \cdots = t_{m-1} = t$, then $t_{m} = t_{m+1} = \cdots = t_{2m-2} = t^2$, then $t_{2m-1} = \cdots = t_{3m-3} = t^3$, and so on.  The result of this procedure is

\begin{theorem}\label{thm:main}
For all $\bfn = (n_1, \cdots, n_m)\in {\mathbb Z}_+^m$ and $d \in \bbN$, let
\[
    h_d(\bfn) = \dim {\left( \mcH^d(M_n) \right)}^{\K{(\bfn)}}.
\] Then for fixed $d$, the limit
\[
    \lim_{n_1 \rightarrow \infty} \cdots \lim_{n_m \rightarrow \infty}
            h_d(n_1, \cdots, n_m)
\] exists.  Let the limiting value be $h_d$.  Then
\[
    h_d = \lim_{n \rightarrow \infty} \dim {(\mcA^m_n[d])}^{\GL_n}.
\]
\end{theorem}
\begin{proof}
After the specialization we obtain:
\begin{eqnarray*}
\prod_{k=1}^\infty
\frac{1}{1- (   \underbrace{t^{ k} + \cdots + t^{ k}}_{m-1 \mbox{ copies }} +
                \underbrace{t^{2k} + \cdots + t^{2k}}_{m-1 \mbox{ copies }} + \cdots)}
&=& \prod_{k=1}^\infty \frac{1}{1 - (m-1) (t^k + t^{2k} + \cdots )} \\
&=& \prod_{k=1}^\infty \frac{1}{1 - (m-1) \frac{t^k}{1-t^k} } \\
&=& \prod_{k=1}^\infty \frac{1}{ \frac{1-t^k - (m-1) t^k}{1-t^k} } \\
&=& \prod_{k=1}^\infty \frac{1-t^k}{1-m \, t^k}.
\end{eqnarray*}
The significance of this calculation is that it allows for another interpretation of sums of Littlewood-Richardson coefficients.  The rest of the proof is identical to the $m=2$ case in the preceding discussion. \qed
\end{proof}

\subsection{A bigraded algebra and a specialization of the main formula}

As before, the group $\GL_n$ acts on $M_n$ by conjugation, and then in turn acts diagonally on $M_n \oplus M_n$.  That is, given $(X,Y) \in M_n \oplus M_n$, and $g \in \GL_n$, we have $g \cdot (X,Y) = (g X g^{-1}, g Y g^{-1})$.  We then obtain an action on $\bbC[M_n \oplus M_n]$.

We have already observed that the algebra $\bbC[M_n \oplus M_n]$ is bigraded.  That is, let $\bbC[M_n \oplus M_n](i,j)$ denote the homogenous polynomial functions on $M_n \oplus M_n$ of degree $i$ in the first copy of $M_n$ and degree $j$ in the second copy of $M_n$.  Let $a$ and $b$ be positive integers.  We set
\[ \bbC[M_n \oplus M_n; (a,b)](a i,b j) = \bbC[M_n \oplus M_n](i, j) \]
with the other components zero.
Put another way, $\bbC[M_n \oplus M_n; (a,b)]$ is the algebra of polynomial function on $M_n \oplus M_n$ together with the bi-gradation defined by $(a,b)$ times the usual degree.  As before we consider the infinite tensor product \[
    \mcB_n = \bigotimes_{a=1}^\infty \bigotimes_{b=1}^\infty \bbC[M_n \oplus M_n; (a,b)].
\]

Next, note the following, obvious, identity:
\[
    \prod_{k=1}^\infty \frac{1}{1- \frac{x^k y^k}{(1-x^k)(1-y^k)}}
    = \prod_{k=1}^\infty \frac{(1-x^k)(1-y^k)}{1-(x^k+y^k)}.
\]
We observe that this is a specialization of the product in the main formula.  Specifically, let $q^i t^j=z_s$ where $z_s$ is given as the $(i,j)$ entry in the following table.

\begin{center}
\begin{tabular}{l| c c c c c c}
$q^it^j$       & $q$         & $q^2$ & $q^3$ & $q^5$ & $q^6$ & $\cdots$  \\
\hline
$t$       & $z_1$       & $z_2$ & $z_4$ & $z_7$ & $z_{11}$ &$\cdots$   \\
$t^2$     & $z_3$       & $z_5$ & $z_8$ & $z_{12}$ & $z_{17}$&$\cdots$  \\
$t^3$     & $z_6$       & $z_9$ & $z_{13}$ &$z_{18}$ &  $z_{24}$ &$\cdots$  \\
$t^4$     & $z_{10}$    & $z_{14}$ &$z_{19}$ & $z_{25}$ &  $z_{32}$ & $\cdots$ \\
$t^5$     & $z_{15}$    & $z_{20}$& $z_{26}$ & $z_{33}$ & $z_{41}$  &$\cdots$  \\
$\vdots$     & $\vdots$    &$\vdots$ &$\vdots$  &$\vdots$  &$\vdots$   &$\ddots$  \\
\end{tabular}
\end{center}

Then,
\begin{eqnarray*}
\displaystyle\prod_{k=1}^{\infty}\frac{1}{1-(z_1^{k}+z_2^{k}+z_3^{k}+\cdots)}&=\prod_{k=1}^{\infty}\frac{1}{1-\displaystyle\sum_{i,j\ge 1}(q^{i} t^{j})^k}\\
&=\prod_{k=1}^{\infty}\frac{1}{1-\frac{q^{k}t^{k}}{(1-q^k)(1-t^k)}}\\
&=\prod_{k=1}^{\infty}\frac{(1-q^k)(1-t^k)}{1-(q^k+t^k)}.
\end{eqnarray*}

In this way, the $\GL_n$-invariants in the harmonic polynomials on $M_n \oplus M_n$ may be related to the multigraded algebra structure $\mcB$, as in the singly graded case.

This identity becomes significantly more complicated when generalized to harmonic polynomials on more than two copies of the matrices.  This fact will be the subject of future work.

\section{Some combinatorics related to finite fields}\label{sec:combinatorics}

In this section, we collect remarks of a combinatorial nature that provide a more concrete understanding of the sum of squares that we consider in this paper.

From elementary combinatorics one knows that the infinite product
\[
    \prod_{k=1}^\infty \frac{1}{1-q \; t^k} = \sum_{n=0}^\infty \sum_{\ell=0}^\infty p_{n,\ell} q^\ell t^n
\] where $p_{n,\ell}$ is the number of partitions of $n$ with exactly $\ell$ parts.   When $q$ is specialized to a positive integer the coefficients of this series in $t$ has many interpretations.

\subsection{The symmetric group}

Fix a positive integer $d$, and a tuple of positive integers
$\bfd = (d_1, d_2, \cdots, d_m)$ with $d_1 + \cdots + d_m = d$.  As before, let $S_\bfd$ denote the subgroup of $S_d$ isomorphic to $S_{d_1} \times \cdots \times S_{d_m}$ embedded by letting the $i^{th}$ factor permute the set $J_i$ where $\{J_1, \cdots, J_m \}$ is the partition of $\{1, \cdots, d\}$ into the $m$ contiguous intervals with $|J_i| = d_i$.  That is, $J_1= \{1,2,\cdots,d_1\}$, $J_2 = \{d_1+1, \cdots, d_1+d_2 \}$, etc.

The group $S_\bfd$ acts on $S_d$ by conjugation.  Let the set of orbits be denoted by $\mcO(\bfd)$.  We have

\begin{proposition}  For any $\bfd = (d_1, \cdots, d_m)$,
\[
    |\mcO(\bfd)| = \sum_{\la \vdash m} \; \sum_{\bfmu = (\mu^{(1)}, \cdots, \mu^{(m)})} \left( \c{\la}{\bfmu} \right)^2
\] where the inner sum is over all tuples of partitions with $\mu^{(j)} \vdash d_j$.
\end{proposition}
\begin{proof}  We begin with the Peter-Weyl type decomposition of $\bbC[S_d]$
\[
    \bbC[S_d] \cong \bigoplus_{\la \vdash d}  \; \dual{V_d^\la} \ot V_d^\la.
\]  We then recall that Littlewood-Richardson coefficients describe the branching rule from $S_d$ to $S_\bfd$:
\[
    V^\la = \bigoplus_{\bfmu} \c{\la}{\bfmu} V^{\mu^{(1)}} \ot \cdots \ot V^{\mu^{(m)}}.
\]
Combining the above decompositions, we obtain the result from Schur's Lemma.  \qed
\end{proof}

It is an elementary fact that the $S_d$-conjugacy class of a permutation $\sigma$ in $S_d$ is determined by the lengths of the disjoint cycles of $\sigma$.  A slightly more general statement is that the $S_\bfd$-conjugacy class of $\sigma \in S_d$ is determined by a union of cycles in which each element of a cycle is ``colored'' by colors corresponding to $J_1, \cdots, J_m$.  It is not difficult to write down a proof of this fact, but we omit it here for space considerations.

If one fixes $d$, then the sum over all $d$-compositions, $\sum_\bfd |\mcO(\bfd)|$, may be interpreted as the number of $d$-bead unions of necklaces with each bead colored by $m$ colors.  For example if $m=2$, and $d=4$, the resulting set may be depicted as:
\begin{center}
\includegraphics[width=11cm]{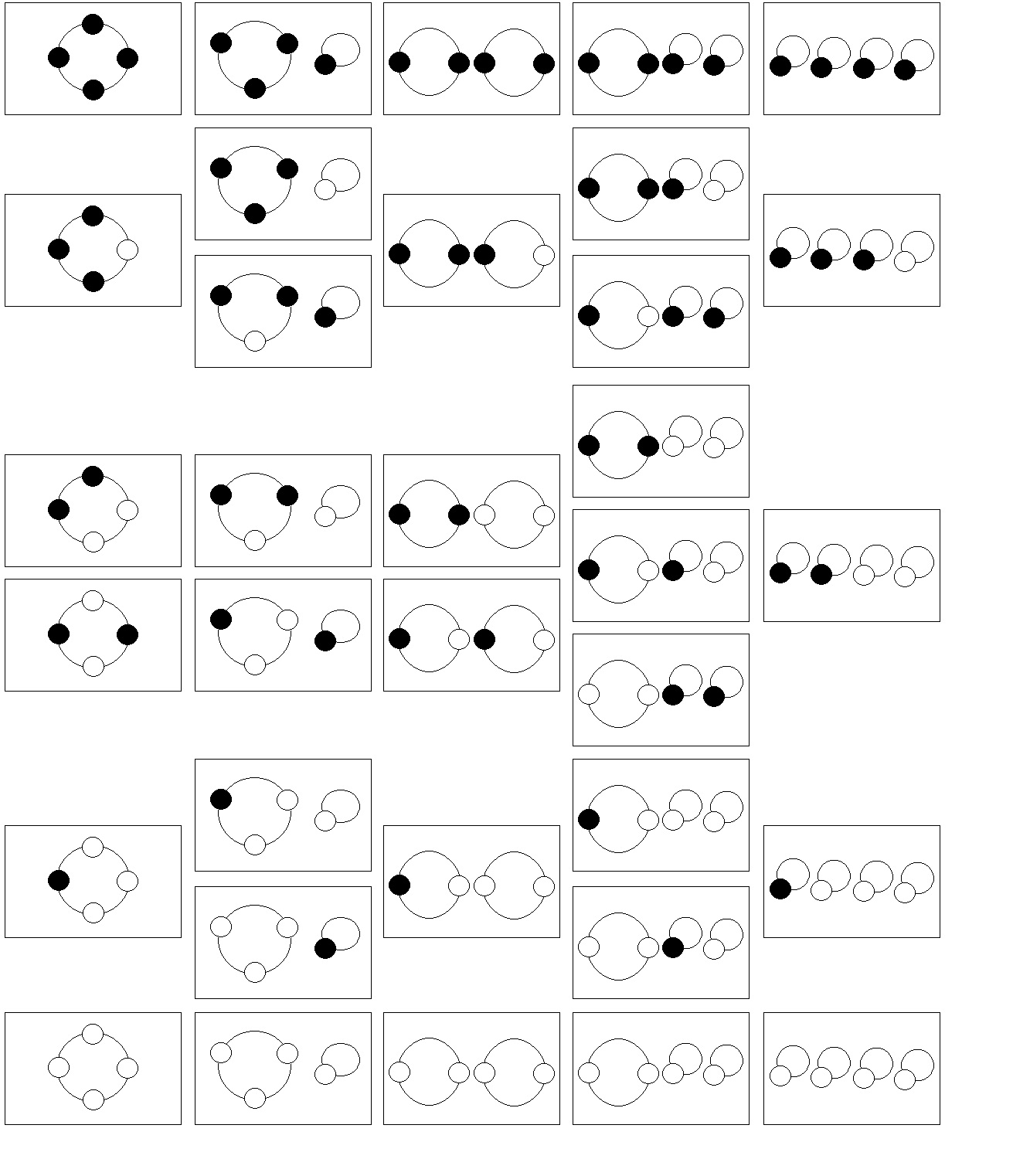}
\end{center}

A single $k$-bead necklace has $\bbZ/ k \bbZ$-symmetry.
If the beads of this necklace are colored with $m$ colors then the resulting colored necklace may have smaller group of symmetries.  Using Polya enumeration (i.e. ``Burnside's Lemma''), one can count such necklaces by the formula
\[
    N_k(m) = \frac{1}{k} \sum_{r|k} \phi(r) m^\frac{k}{r}.
\] where $\phi$ denotes the Euler totient function\footnote{That is, $\phi(r)$ is the number of positive integers relatively prime to $r$.}. The theory forming the underpinnings of the above formula may be put in a larger context of \emph{cycle index polynomials}.  We refer the reader to Doron Zeilberger's survey article (IV.18 of \cite{Gowers}).

The generating function for disjoint unions of such necklaces can be given by the product
\[
    \eta_m(t) = \prod_{k=1}^\infty {\left( \frac{1}{1-t^k} \right)}^{N_k(m)}
\]  That is to say that the number of $d$-bead necklaces, counted up to cyclic symmetry, is equal to the coefficient of $t^d$ in $\eta(t)$.  The main formula specializes to
\begin{equation}\label{eqn:eta}
    \eta_m(t) = \sum_{\la,\bfmu \vdash d} \left( \c{\la}{\bfmu} \right)^2 t^{|\la|},
\end{equation} where the sum is over partitions $\la$ and $m$-tuples of partitions, $\bfmu = (\mu^{(1)}, \cdots, \mu^{(m)})$.

Equation \ref{eqn:eta} begs for an (explicit) bijective proof which is no doubt obtained by merging both the Littlewood-Richardson rule and the Robinson-Schensted-Knuth bijection.  It is likely that more than one ``natural'' bijection exists.

\subsection{The general linear group over a finite field}

Let $p$ denote a prime number, and $v \in \bbZ^+$.  Set $q=p^v$.    Let $\GL_m(q)$ denote the general linear group of invertible $m \times m$ matrices over the field with $q$ elements.  The set, $\mathcal C_m(q)$, of conjugacy classes of $\GL_m(q)$ has a cardinality of note in that the infinite products has the following expansion:
\[
    \prod_{k=1}^\infty \frac{1-t^k}{1-q \, t^k} = \sum_{m=0} |\mathcal C_m(q)| t^m,
\] see \cite{Benson-Feit-Howe, Stong}, and the references within.

Note that from Equation \ref{eqn:eta} we obtain a new formula for the number of conjugacy classes of $\GL_m(q)$.  Namely,
\[
    \left( \prod_{k=1}^\infty (1-t^k) \right) \eta_q(t) =
    \prod_{k=1}^\infty \left(\frac{1}{1-t^k}\right)^{N_k(q)-1}.
\]

Stong reference

\section{Proof of the main formula}\label{sec:main_formula}

Before proceeding, we require some more notation.  Given a partition $\la = (\la_1 \geq \la_2 \geq \cdots)$ with $v_1$ ones, and $v_2$ twos, etc., we will call the sequence $(v_1, v_2, \cdots)$ the \emph{type vector} of $\la$.  Note that $\ell(\la) = \sum v_i$, while the size of $\la$ is $m = |\la| = \sum \la_i = \sum i v_i$.  As is standard, we set
\[
    z_\la = (v_1! 1^{v_1}) (v_2! 2^{v_2}) (v_3! 3^{v_3}) \cdots
\]  It is elementary that the cardinality of a conjugacy class of a permutation with cycle type $\la$ is $\frac{m!}{z_\la}$.  (Equivalently, the centralizer subgroup has order $z_\la$.)

\bigskip
We now prove the main formula.  The product on the left side (LS) may be expanded using the sum of a geometric series, the multinomial theorem, and then the sum-product formula, as follows
\begin{eqnarray*}
LS = \prod_{k=1}^\infty \frac{1}{1-(t_1^k+t_2^k+t_3^k+\cdots)} &=&
    \prod_{k=1}^\infty \sum_{u=0}^\infty (t_1^k + t_2^k + \cdots)^u \\
&=& \prod_{k=1}^\infty \sum_{u=0}^\infty \; \sum_{u_1 + u_2 + \cdots = u}
        \frac{u!}{u_1! u_2! \cdots} t_1^{k u_1} t_2^{k u_2} \cdots \\
&=& \prod_{k=1}^\infty \sum_{u_1, u_2, \cdots}
        \frac{(u_1 + u_2 + \cdots)!}{u_1! u_2! \cdots} t_1^{k u_1} t_2^{k u_2} \cdots \\
&=& \sum_{u_j^{(i)}, \cdots}
        \prod_{i=1}^\infty
            \frac{(u_1^{(i)} + u_2^{(i)} + \cdots)!}{u_1^{(i)}! u_2^{(i)}! \cdots}
                t_1^{i u_1^{(i)}} t_2^{i u_2^{(i)}} \cdots
\end{eqnarray*}
(with all sequences having finite support.)  We will introduce another sequence, ${\bf a} = (a_1, a_2, \cdots)$ and extract the coefficient of ${\bft}^{\bf a} = \prod t_i^{a_i}$ in the above formal expression to obtain
\[
LS = \sum_{{\bf a}}
    \left( \sum_{u_j^{(i)}: \forall j,\, \sum_i i u_j^{(i)} = a_j} \prod_{i=1}^\infty \frac{(u_1^{(i)} + u_2^{(i)} + \cdots)!}{u_1^{(i)}! u_2^{(i)}! \cdots} \right) \bft^{\bf a}
\]  The above is not such a complicated expression, although these formal manipulations may, at first, seem daunting.  Observe that the sequence $u_j^{(1)}, u_j^{(2)}, \cdots$ with $\sum_i i u_j^{(i)} = a_j$ encodes a partition of $a_j$ with the number $i$ occurring exactly $u_j^{(i)}$ times, while $\sum_{i \geq 1} u_j^{(i)}$ is the length of the partition.  We will call this partition $\mu^{(j)}$.  That is to say, $\mu^{(j)}$ has type vector $(u_j^{(1)}, u_j^{(2)}, \cdots)$.  The coefficient of $\bft^{\bf a}$ may be rewritten as a sum over all double sequences $u_j^{(i)}$ such that for all $j$, $\sum_i i u_j^{(i)} = a_j$ of
\begin{equation}\label{eqn:product}
\prod_{i=1}^\infty  \frac{(u_1^{(i)} + u_2^{(i)} + \cdots)!}{u_1^{(i)}! u_2^{(i)}! \cdots} =
\prod_{i=1}^\infty  \frac{(u_1^{(i)} + u_2^{(i)} + \cdots)! i^{\sum u_j^{(i)}}}{
(u_1^{(i)}! i^{u_1^{(i)}}) (u_2^{(i)}! i^{u_2^{(i)}}) \cdots}.
\end{equation}
Given a (finitely supported) sequence of partitions $\bfmu = (\mu^{(1)}, \mu^{(2)}, \cdots)$, we denote the partition obtained from the (sorted) concatenation of all $\mu^{(j)}$ by
$\cup \mu^{(j)}$.  Thus, if $\mu^{(j)}$ has type vector $(u_1^{(i)}, u_2^{(i)}, \cdots)$ then the number of $i$'s in $\cup \mu^{(j)}$ is $u_1^{(i)} + u_2^{(i)} + \cdots$.
It therefore follows that the numerator of the righthand side of Equation \ref{eqn:product} is $z_\la$ when $\la = \cup \mu^{(j)}$.  The denominator can easily be seen to be $z_{\mu^{(j)}}$.  From this observation we obtain
\[
    LS = \sum_{\bfmu} \frac{z_{\cup_{j=1}^\infty \mu^{(j)}}}{\prod_{j=1}^\infty z_{\mu^{(j)}}} z_1^{|\mu^{(1)}|} z_2^{|\mu^{(2)}|} z_3^{|\mu^{(3)}|} \cdots.
\]

\subsection{An application of the Hall scalar product}

Let $\Lambda_n$ denote the $S_n$-invariant polynomials (over $\bbC$ as usual) in the indeterminates $x_1, \cdots, x_n$, let $\Lambda[\bfx] = \lim_{\leftarrow} \Lambda_n$ denote the inverse limit.  Thus, $\Lambda[\bfx]$ is the algebra of symmetric functions.  For a non-negative integer partition $\la$, we let $s_\la(\bfx)$ denote the Schur function.  That is, for each $n$, $s_\la(\bfx)$ projects to the polynomial $s_\la(x_1, \cdots, x_n)$ which, as a function on the diagonal matrices, coincides with the character of the $\GL_n$-irrep $\F{\la}{n}$.  The set $\{s_\la(\bfx): \la \vdash d \}$ is a $\bbC$-vector space basis of the homogeneous degree $d$ symmetric functions.  We will define a non-degenerate symmetric bilinear form $\form{\cdot,\cdot}$ by declaring $\form{s_\al(\bfx), s_\be(\bfx)} = \de_{\al \be}$ for all non-negative integer partitions $\al$, $\be$.  This form is the \emph{Hall scalar product}.

Given an integer $m$, define $p_m(\bfx) = x_1^m + x_2^m + \cdots$ to denote the power sum symmetric function.  Given $\nu \vdash N$, set $p_\nu(\bfx) = \prod p_{\nu_j}(\bfx)$.   We remark that the left side of the main formula is easily seen to be $\sum_\nu p_\nu(\bft)$.

It is a consequence of Schur-Weyl duality that the coefficients of the Schur function expansion
\[
    p_\nu(\bfx) = \sum_\la \chi^\la(\nu) s_\la(\bfx)
\] are the characters of the $S_N$-irrep indexed by $\la$ evaluated at any permutation with cycle type $\nu$.  It is a standard exercise to see, from the orthogonality of the character table, that $p_\nu(\bfx)$ are an orthogonal basis of $\Lambda[\bfx]$, and
$\form{p_\nu(\bfx), p_\nu(\bfx)} = z_\nu$.

We next consider another set of indeterminates, $\bfy = y_1, y_2, \cdots$.
Set $\Lambda[\bfx,\bfy] = \Lambda[\bfx]\ot\Lambda[\bfy]$.
The Hall scalar product extends to $\Lambda[\bfx,\bfy]$ in the standard way as
\[
    \form{ f(\bfx)\ot g(\bfy), f'(\bfx)g'(\bfy) } =
    \form{f(\bfx), f'(\bfx)} \form{g(\bfy), g'(\bfy)}.
\] where $f(\bfx),f'(\bfx) \in \Lambda[\bfx]$ and $g(\bfy), g'(\bfy) \in \Lambda[\bfy]$.
(We then extend by linearity to all of $\Lambda[\bfx,\bfy]$.)

The character theoretic consequence of ($\GL_n, \GL_k$)-Howe duality is the Cauchy identity:
\begin{equation}\label{eqn:scalar1}
    \prod_{i,j=1}^\infty \frac{1}{1-x_i y_j} = \sum_\la s_\la(\bfx) s_\la(\bfy)
\end{equation} (in the infinite sets of variables).  In fact, for any pair of dual bases, $a_\la, b_\la$, with respect to the Hall scalar product, one has $\prod_{i,j=1}^\infty \frac{1}{1-x_i y_j} = \sum_\la a_\la(\bfx) b_\la(\bfy)$.  From this fact one obtains
\begin{equation}\label{eqn:scalar2}
    \prod_{i,j=1}^\infty \frac{1}{1-x_i y_j} = \sum_\la p_\la(\bfx) p_\la(\bfy)/z_\la.
\end{equation}

From our point of view, we will expand the following scalar product
\begin{equation}\label{eqn:scalar3}
    \form{
        \prod_{k=1}^\infty \prod_{i,j=1}^\infty \frac{1}{1-x_i y_j t_k},
        \prod_{i,j=1}^\infty \frac{1}{1-x_i y_j}
    }
\end{equation} in two different ways, corresponding to Equations \ref{eqn:scalar1} and \ref{eqn:scalar2}.

First by homogeneity of the Schur function and Cauchy's identity
\[
     \prod_{k=1}^\infty \sum_\mu s_\mu(\bfx) s_\mu(\bfy) t_k^{|\mu|} =
    \sum_{\bfmu = (\mu^{(1)}, \mu^{(2)}, \cdots)}
        \prod_j s_{\mu^{(j)}}(\bfx) \prod_j s_{\mu^{(j)}}(\bfy) t_1^{|\mu^{(1)}|} t_2^{|\mu^{(2)}|} \cdots.
\]  Since the multiplication of characters is the character of the tensor product of the corresponding representations, we have $s_\al s_\be = \sum \c{\ga}{\al \be} s_\ga$ in the $\bfx$ (resp. $\bfy$) variables.  Expanding the above product gives
\[
    \form{
        \prod_{k=1}^\infty \prod_{i,j=1}^\infty \frac{1}{1-x_i y_j t_k},
        \prod_{i,j=1}^\infty \frac{1}{1-x_i y_j}
    } = \sum_{\bfmu} \left( \c{\la}{\bfmu} \right)^2
            t_1^{|\mu^{(1)}|} t_2^{|\mu^{(2)}|} \cdots.
\]

Secondly, the scalar product in (\ref{eqn:scalar3}) may be expressed as
\[
    \form{\prod_{k=1}^\infty \; \sum_{\mu^{(k)}}
    p_{\mu^{(k)}}(\bfx)
    p_{\mu^{(k)}}(\bfy)/{z_{\mu^{(k)}}} \; t^{|\mu^{(k)}|},
    \sum_\la p_\la(\bfx) p_\la(\bfy)/z_\la }.
\]
We observe that $\prod_{k=1}^\infty p_{\mu^{(k)}} = p_\la$ where
$\la = \cup_{k=1}^\infty \mu^{(k)}$.

\subsection{A remark from ``Macdonald's book''}

The results presented in this paper emphasize describing the cardinality of an orbit space by a sum of Littlewood-Richardson coefficients.  It is important to note, however, that the main formula can be written simply as
\[
    \sum_\nu p_\nu(\bfx) = \sum_\bfmu \left( \sum_\la {\left(\c{\la}{\bfmu}\right)}^2 \right)
    x_1^{\mu^{(1)}} x_2^{\mu^{(2)}} x_1^{\mu^{(3)}} \cdots
\]
So from the point of view of \cite{Macdonald}, one realizes that the main formula is simply a way of expanding $\sum p_\nu(\bfx)$.  With remark in mind, we recall the ``standard'' viewpoint.

For a non-negative integer partition $\de = (\de_1 \geq \de_2 \geq \cdots)$ let $x^\de = x_1^{\de_1} x_2^{\de_2} x_3^{\de_3} \cdots$.  The monomial symmetric function, $m_\de(\bfx)$, is the sum over the orbit obtained by all permutations of the variables.  The monomial symmetric functions are a basis for the algebra $\Lambda$.

The question becomes obtaining an expansion of $\sum_\ga p_\ga(\bfx)$ into monomial symmetric functions.  This question is answered immediately by observing the expansion of $p_\ga$ into monomial symmetric functions, which can be found in \cite{Macdonald} on page 102 of Chapter I, Section 6.  For partitions $\ga$ and $\de$ define $L(\ga, \de)$ by the expansion
\[
    p_\ga(\bfx) = \sum_\de L_{\ga \de} m_\de(\bfx).
\]

We next provide a combinatorial description of $L_{\ga \de}$.  Let $\ga$ denote a partition of length $\ell$.  Given an integer valued function, $f$, defined on $\{1,2,3, \cdots, \ell\}$, set
\[
    f(\ga)_i = \sum_{j: f(j) = i} \nu_j
\] for each $i \geq 1$.

The sequence $(f(\ga)_1, f(\ga)_2, f(\ga)_3, \cdots)$ does not have to be weakly decreasing.  For example, if $\ga = (1,1,1)$ and $f(1)=1$, $f(2)=4$ and $f(3)=4$ then $f(\ga)_1=1$, $f(\ga)_4=2$ and $f(\ga)_k = 0$ for all $k\neq 1,4$.  However, often this sequences does define a partition.  We have

\begin{proposition}\label{prop_Macdonald}
$L_{\ga, \de}$ is equal to the number of functions $f$ such that $f(\ga) = \de$.
\end{proposition}
\begin{proof}
See \cite{Macdonald} Proposition I (6.9) \qed
\end{proof}
From Proposition \ref{prop_Macdonald} and the main formula, we obtain

\begin{corollary}  Given a partition $\de$, the cardinality of
\[
    \left\{ f | \mbox{ for some parition $\ga$,} \ f(\ga) = \de \; \right\}
\] is equal to
\[
    \sum_\la \sum_\bfmu \left( \c{\la}{\bfmu} \right)^2
\] where the sum is over all $\bfmu$ such that $|\mu^{(j)}| = \de_j$ for all $j$ and $|\la| = |\bfmu|$.
\end{corollary}

\def\cprime{$'$} \def\cprime{$'$}


\begin{bibdiv}
\bibliographystyle{alphabetic}
\begin{biblist}
\bib{Benson-Feit-Howe}{article}{
   author={Benson, David},
   author={Feit, Walter},
   author={Howe, Roger},
   title={Finite linear groups, the Commodore 64, Euler and Sylvester},
   journal={Amer. Math. Monthly},
   volume={93},
   date={1986},
   number={9},
   pages={717--719},
   issn={0002-9890},
   review={\MR{863974 (87m:05011)}},
   doi={10.2307/2322289},
}
\bib{Chevalley}{article}{
   author={Chevalley, C.},
   title={Sur certains groupes simples},
   language={French},
   journal={T\^ohoku Math. J. (2)},
   volume={7},
   date={1955},
   pages={14--66},
   issn={0040-8735},
   review={\MR{0073602 (17,457c)}},
}
\bib{Drensky}{article}{
   author={Drensky, Vesselin},
   title={Computing with matrix invariants},
   journal={Math. Balkanica (N.S.)},
   volume={21},
   date={2007},
   number={1-2},
   pages={141--172},
   issn={0205-3217},
   review={\MR{2350726 (2008m:16045)}},
}
\bib{Fulton}{book}{
   author={Fulton, William},
   title={Young tableaux},
   series={London Mathematical Society Student Texts},
   volume={35},
   note={With applications to representation theory and geometry},
   publisher={Cambridge University Press},
   place={Cambridge},
   date={1997},
   pages={x+260},
   isbn={0-521-56144-2},
   isbn={0-521-56724-6},
   review={\MR{1464693 (99f:05119)}},
}
\bib{GW}{book}{
   author={Goodman, Roe},
   author={Wallach, Nolan R.},
   title={Symmetry, representations, and invariants},
   series={Graduate Texts in Mathematics},
   volume={255},
   publisher={Springer},
   place={Dordrecht},
   date={2009},
   pages={xx+716},
   isbn={978-0-387-79851-6},
   review={\MR{2522486}},
   doi={10.1007/978-0-387-79852-3},
}
\bib{Gowers}{collection}{
   title={The Princeton companion to mathematics},
   editor={Gowers, Timothy},
   editor={Barrow-Green, June},
   editor={Leader, Imre},
   publisher={Princeton University Press},
   place={Princeton, NJ},
   date={2008},
   pages={xxii+1034},
   isbn={978-0-691-11880-2},
   review={\MR{2467561 (2009i:00002)}},
}
\bib{Harris}{article}{
   author={Harris, Pamela E.},
   title={On the adjoint representation of $\germ s\germ l_n$ and the
   Fibonacci numbers},
   language={English, with English and French summaries},
   journal={C. R. Math. Acad. Sci. Paris},
   volume={349},
   date={2011},
   number={17-18},
   pages={935--937},
   issn={1631-073X},
   review={\MR{2838238}},
   doi={10.1016/j.crma.2011.08.017},
}
\bib{Hesselink}{article}{
   author={Hesselink, Wim H.},
   title={Characters of the nullcone},
   journal={Math. Ann.},
   volume={252},
   date={1980},
   number={3},
   pages={179--182},
   issn={0025-5831},
   review={\MR{593631 (82c:17004)}},
   doi={10.1007/BF01420081},
}
\bib{Howe-Lee-1}{article}{
   author={Howe, Roger},
   author={Lee, Soo Teck},
   title={Bases for some reciprocity algebras. I},
   journal={Trans. Amer. Math. Soc.},
   volume={359},
   date={2007},
   number={9},
   pages={4359--4387},
   issn={0002-9947},
   review={\MR{2309189 (2008j:22017)}},
   doi={10.1090/S0002-9947-07-04142-6},
}
\bib{Howe-Lee-2}{article}{
   author={Howe, Roger},
   author={Lee, Soo Teck},
   title={Why should the Littlewood-Richardson rule be true?},
   journal={Bull. Amer. Math. Soc. (N.S.)},
   volume={49},
   date={2012},
   number={2},
   pages={187--236},
   issn={0273-0979},
   review={\MR{2888167}},
   doi={10.1090/S0273-0979-2011-01358-1},
}
\bib{HTW}{article}{
   author={Howe, Roger},
   author={Tan, Eng-Chye},
   author={Willenbring, Jeb F.},
   title={Stable branching rules for classical symmetric pairs},
   journal={Trans. Amer. Math. Soc.},
   volume={357},
   date={2005},
   number={4},
   pages={1601--1626},
   issn={0002-9947},
   review={\MR{2115378 (2005j:22007)}},
   doi={10.1090/S0002-9947-04-03722-5},
}
\bib{Kostant}{article}{
   author={Kostant, Bertram},
   title={Lie group representations on polynomial rings},
   journal={Amer. J. Math.},
   volume={85},
   date={1963},
   pages={327--404},
   issn={0002-9327},
   review={\MR{0158024 (28 \#1252)}},
}
\bib{Macdonald}{book}{
   author={Macdonald, I. G.},
   title={Symmetric functions and Hall polynomials},
   series={Oxford Mathematical Monographs},
   edition={2},
   note={With contributions by A. Zelevinsky;
   Oxford Science Publications},
   publisher={The Clarendon Press Oxford University Press},
   place={New York},
   date={1995},
   pages={x+475},
   isbn={0-19-853489-2},
   review={\MR{1354144 (96h:05207)}},
}
\bib{Procesi-AMS}{article}{
   author={Procesi, Claudio},
   title={The invariants of $n\times n$ matrices},
   journal={Bull. Amer. Math. Soc.},
   volume={82},
   date={1976},
   number={6},
   pages={891--892},
   issn={0002-9904},
   review={\MR{0419490 (54 \#7511)}},
}
\bib{Procesi-AIM}{article}{
   author={Procesi, C.},
   title={The invariant theory of $n\times n$ matrices},
   journal={Advances in Math.},
   volume={19},
   date={1976},
   number={3},
   pages={306--381},
   issn={0001-8708},
   review={\MR{0419491 (54 \#7512)}},
}
\bib{Stong}{article}{
   author={Stong, Richard},
   title={Some asymptotic results on finite vector spaces},
   journal={Adv. in Appl. Math.},
   volume={9},
   date={1988},
   number={2},
   pages={167--199},
   issn={0196-8858},
   review={\MR{937520 (89c:05007)}},
   doi={10.1016/0196-8858(88)90012-7},
}
\bib{Wallach-Willenbring}{article}{
   author={Wallach, N. R.},
   author={Willenbring, J.},
   title={On some $q$-analogs of a theorem of Kostant-Rallis},
   journal={Canad. J. Math.},
   volume={52},
   date={2000},
   number={2},
   pages={438--448},
   issn={0008-414X},
   review={\MR{1755786 (2001j:22020)}},
   doi={10.4153/CJM-2000-020-0},
}
\bib{Willenbring}{article}{
   author={Willenbring, Jeb F.},
   title={Stable Hilbert series of $S(\germ g)^\K$ for classical
   groups},
   journal={J. Algebra},
   volume={314},
   date={2007},
   number={2},
   pages={844--871},
   issn={0021-8693},
   review={\MR{2344587 (2008j:22021)}},
   doi={10.1016/j.jalgebra.2007.04.014}
}
\end{biblist}
\end{bibdiv}
\end{document}